\documentclass[11pt]{article}
\usepackage{indentfirst}

\usepackage{float}
\usepackage{booktabs}
\usepackage{makecell}
\usepackage{epsfig}
\usepackage{url}
\usepackage{ulem}
\usepackage{amsmath}
\usepackage{amsfonts}
\usepackage{color}
\usepackage{amssymb}
\usepackage{graphicx,amsfonts,mathrsfs}
\usepackage{epstopdf}
\usepackage{enumerate}
\usepackage[all]{xy}
\usepackage{amsthm}
\usepackage{array}

\def\id{{\rm \textsf{id}}}
\def\Fix{{\rm \textsf{fix}}}
\def\Aut{{\rm \textsf{Aut}}}

\allowdisplaybreaks

\begin{document}
\begin{sloppypar}

\newtheorem{theorem}{Theorem}[section]
\newtheorem{remark}[theorem]{Remark}
\newtheorem{problem}[theorem]{Problem}
\newtheorem{corollary}[theorem]{Corollary}
\newtheorem{definition}[theorem]{Definition}
\newtheorem{conjecture}[theorem]{Conjecture}
\newtheorem{question}[theorem]{Question}

\newtheorem{lemma}[theorem]{Lemma}
\newtheorem{proposition}[theorem]{Proposition}
\newtheorem{quest}[theorem]{Question}
\newtheorem{example}[theorem]{Example}
\newtheorem{observation}[theorem]{Observation}
\newcommand{\pp}{{\it p.}}
\newcommand{\de}{\em}

\title{  {Generalized Cayley graphs and perfect code}\thanks{This research was supported by  NSFC (No. 12071484), Hunan Provincial Natural Science Foundation (2020JJ4675).
 E-mail addresses: liaoqianfen@163.com(Q. Liao),  wjliu6210@126.com(W. Liu, corresponding author).}}

\author{Qianfen Liao$^1$, Weijun Liu$^{1,2\dagger}$\\
{\small $^{1}$ School of Mathematics and Statistics, Central South University} \\
{\small New Campus, Changsha, Hunan, 410083, P.R. China. }\\
{\small $^2$ College of General Education, Guangdong University of Science and Technology,} \\
{\small Dongguan, Guangdong, 523083, P.R. China}
}
\maketitle

\vspace{-0.5cm}

\begin{abstract}
Perfect code in Cayley graphs and Cayley sum graphs is studied extensively in recent years.
In this paper, we consider perfect code in generalized Cayley graphs.
 \end{abstract}

\section{Introduction}

Given a graph $\Gamma$,  a subset $C_1$ of $V(\Gamma)$  is called a perfect code in $\Gamma$ if $C_1$ is an independent set and every vertex in $(\Gamma)\setminus C_1$ is adjacent  to exactly one vertex in $C_1$.
A subset $C_2$ of $V(\Gamma)$ is said to be a total perfect code  in $\Gamma$ if every vertex of $\Gamma$ has exactly one neighbor in $C_2$.
(In particular, $C_2$ induces a matching in $\Gamma$ and so $|C_2|$ is even.)

Given a finite group $G$ and an  automorphism $\alpha$ of $G$ with $\alpha^2=\id$, define subsets $\omega_\alpha(G)=\{\alpha(g^{-1})g\mid g\in G\}$, $\Omega_\alpha(G)=\{g\in G\setminus \omega_\alpha(G) \mid  \alpha(g)=g^{-1}\}$ and $\mho_\alpha(G)=\{g\in G \mid \alpha(g)\neq g^{-1}\}$.
A subset $S$ of $G$ is called a generalzied Cayley subset induced by $\alpha$ if it satisfies $S\cap \omega_\alpha(G)=\emptyset$ and $\alpha(S)=S^{-1}$.
The graph with vertices $G$ and edges $\{\{g, h\}\mid \alpha(g^{-1})h\in S\}$ is called the generalized Cayley graph of $G$ with respect to the ordered pair $(S,\alpha)$, and denoted by $GC(G, S, \alpha)$.
In particular, if $\alpha=\id$, the corresponding graph is Cayley graph.
To differ from the Cayley graph, we only discuss the case of $\alpha$ is involution of $\Aut(G)$ throughout this paper.

We call a subgroup is a (total) perfect code of $G$ with respect  to $\alpha$ if $H$ is subgroup (total) perfect code of some generalized Cayley graph of $G$ induced by $\alpha$.

In this paper, we study perfect code in generalized Cayley graphs.
All groups appeared among this paper is finite.

\section{}

\begin{lemma}\label{lem1}
Let $S=\{s_1, \cdots, s_r\}$ be a generalized Cayley subset of $G$ induced by the involutory automorphism $\alpha$ and $X$ a subset of $G$.
The following conditions are equivalent:
\begin{enumerate}[{\rm(i)}]

\item each vertex in $GC(G, S, \alpha)$ is adjacent to at most one vertex in $X$;

\item for each two distinct element $s_1$ and $s_2$ in $S$, $\alpha(X)s_1\cap \alpha(X)s_2=\emptyset$;

\item $(\alpha(X^{-1})\alpha(X))\cap(SS^{-1})=\{e\}$.

\end{enumerate}
\end{lemma}
\begin{proof}
\textbf{$(\rm i)\Rightarrow (\rm ii)$}:
Assume that $\alpha(X)s_1\cap \alpha(X)s_2\neq \emptyset$ for some different elements $s_1$ and $s_2$ in $S$.
Then there exists elements $x_1$ and $x_2$ in $X$ such that $\alpha(x_1)s_1=\alpha(x_2)s_2$.
It implies
\begin{equation*}
\alpha(\alpha(x_1)s_1)^{-1}x_2=\alpha(s_1^{-1})x_1^{-1}x_2=\alpha(s_2^{-1})\in S.
\end{equation*}
Thus, the vertex $\alpha(x_1)s_1$ is adjacent to $x_2$.
Observe that $\alpha(x_1)s_1$ is also adjacent to $x_1$ as $\alpha(x_1^{-1})\alpha(x_1)s_1=s_1\in S$.
From the assumption, we have $x_1=x_2$ and this yields that $s_1=s_2$, which is a contradiction.

\textbf{$(\rm ii)\Rightarrow (\rm iii)$}:
This is clear from the fact that, $\alpha(x_1)s_1=\alpha(x_2)s_2$ is equivalent to $\alpha(x_1^{-1})\alpha(x_2)=s_1s_2^{-1}$, where $x_1, x_2\in X$ and $s_1, s_2\in S$.

\textbf{$(\rm iii)\Rightarrow (\rm i)$}:
Assume that there exists element $g$ such that $g$ is adjacent to both $x_1$ and $x_2$, where $x_1$ and $x_2$ are different elements in $X$.
Then $\alpha(x_1^{-1})g$ and $\alpha(x_2^{-1})g$ are belong to $S$.
As a result,
\begin{equation*}
\alpha(x_1^{-1})g(\alpha(x_2^{-1})g)^{-1}=\alpha(x_1^{-1})\alpha(x_2)\in \alpha(X^{-1})\alpha(X)\cap SS^{-1},
\end{equation*}
which leads to a contradiction.
\end{proof}

\begin{lemma}\label{lem2}
Let $S$ be a generalized Cayley subset of $G$ induced by the involutory automorphism $\alpha$ and $X$ a subset of $G$.
In $GC(G, S, \alpha)$, each vertex in $G\setminus X$ is adjacent to at least one vertex in $X$ if and only if $G\setminus X\subseteq \cup_{s\in S}(\alpha(X)s)$.
\end{lemma}

\begin{lemma}\label{lem3}
Let $S$ be a generalized Cayley subset of $G$ induced by the involutory automorphism $\alpha$ and $X$ a subset of $G$.
$X$ is an independent set of $GC(G, S, \alpha)$ if and only if $X\cap \alpha(X)s=\emptyset$ for each $s\in S$, that is $\alpha(X^{-1})X\cap S=\emptyset$.
\end{lemma}

Based on the statements above, we give the equivalent condition of a subset to be perfect code.

\begin{lemma}\label{lem4}
Let $S=\{s_1, \cdots, s_r\}$ be a generalized Cayley subset of $G$ induced by the involutory automorphism $\alpha$ and $X$ a subset of $G$.
The following conditions are equivalent:
\begin{enumerate}[{\rm(i)}]

\item $X$ is a perfect code of $GC(G, S, \alpha)$;

\item $\{X, \alpha(X)s_1, \cdots, \alpha(X)s_r\}$ is a partition of $G$;

\item $|G|=|X|(r+1)$, $\alpha(X^{-1})X\cap S=\emptyset$ and $(\alpha(X^{-1})\alpha(X))\cap(SS^{-1})=\emptyset$.

\end{enumerate}
\end{lemma}
\begin{proof}
Let $X$ be a perfect code of $GC(G, S, \alpha)$.
Since each element in $G\setminus X$ is adjacent to exactly one vertex in $X$ and $X$ is an independent set, by Lemmas \ref{lem1} to \ref{lem4}, we obtain  $\rm{(i)}$ is equivalent to  $\rm{(iii)}$.

$\rm{(ii)}\Leftrightarrow \rm{(iii)}$: $\{X, \alpha(X)s_1, \cdots, \alpha(X)s_r\}$ construct a partition of $G$ is equivalent to $|G|=|X|(r+1)$, and for different integers $i$ and $j$,
$X\cap \alpha(X)s_i=\emptyset$ and  $\alpha(X)s_i\cap \alpha(X)s_j=\emptyset$, where $1\leq i, j\leq r$.
By Lemma \ref{lem1} and \ref{lem3}, $X\cap \alpha(X)s_i=\emptyset$ is equivalent to  $\alpha(X^{-1})X\cap S=\emptyset$, and
$\alpha(X)s_i\cap \alpha(X)s_j=\emptyset$ is equivalent to $(\alpha(X^{-1})\alpha(X))\cap(SS^{-1})=\emptyset$.
\end{proof}

Especially, we find that if a subgroup $H$ of $G$ is perfect code of $GC(G, S, \alpha)$, then it is fixed by $\alpha$.

\begin{proposition}\label{pro1}
Let $H$ be a subgroup of $G$.
If $H$ is subgroup perfect code of  graph $GC(G, S, \alpha)$, then $S\cap \alpha(H)=\emptyset$.
Further, we have $\alpha(H)=H$.
\end{proposition}
\begin{proof}
Let $S=\{s_1, \cdots, s_r\}$.
Firstly, by lemma \ref{lem4}, $H\cap \alpha(H)s_i=\emptyset$ for each $s_i\in S$.
We prove the result by negation.
Assume $S\cap \alpha(H)\neq\emptyset$ and let $s\in S\cap \alpha(H)$.
Then $\alpha(s^{-1})\in H\cap \alpha(H)\alpha(s^{-1})$, which is a contradiction as $\alpha(s^{-1})\in S$.
Hence, $S\cap \alpha(H)=\emptyset$, which indicates that $\alpha(H)\cap\alpha(H)s_i=\emptyset$ for each $s_i\in S$.
Recall that $G=H\cup \alpha(H)s_1\cup \cdots \cup\alpha(H)s_r$, we have $\alpha(H)\subseteq H$.
\end{proof}

\begin{corollary}\label{cor1}
Let $S$ be a generalized Cayley subset of $G$ induced by the involutory automorphism $\alpha$.
Then the subgroup $H$ of $G$ is  perfect code of $GC(G, S, \alpha)$ if and only if $\alpha(H)=H$ and $\{e\}\cup S$ is a set of coset representatives of $H$.
\end{corollary}

\begin{remark}
 Recall that $\alpha(S)=S^{-1}$.
If $H$  is  subgroup perfect code of $GC(G, S, \alpha)$, then
\begin{eqnarray*}
G
&=&H\cup Hs_1\cup\cdots \cup Hs_r\\
&=&\alpha(H)\cup \alpha(Hs_1)\cup\cdots \cup \alpha(Hs_r)\\
&=&H\cup H\alpha(s_1)\cup\cdots \cup H\alpha(s_r)\\
&=&H\cup Hs_1^{-1}\cup\cdots \cup Hs_r^{-1}\\
&=&H\cup s_1H\cup\cdots \cup s_rH.
\end{eqnarray*}
Thus, in Corollary \ref{cor1}, $\{e\}\cup S$ is not only  a set of right coset representatives of $H$, but also a set of left coset representatives of $H$.
\end{remark}

\begin{definition}
Given a group $G$, a subgroup $H$ of $G$, and an involution  $\alpha\in \Aut(G)$ , we call a subset $T$ of $G$ a generalized Cayley transversal of $H$ in $G$ if $T$ containing identity element $e$, $T\setminus \{e\}$ is a generalized Cayley subset induced by $\alpha$ and $T$ is a set of coset representatives of $H$ in $G$.
\end{definition}

\begin{remark}
Note that for $s\in S$ with $\alpha(s)\neq s^{-1}$, $\{s,e\}$ and $\{s, \alpha(s)s\}$ are edges of $GC(G, S, \alpha)$.
Thus if  $H$ is a subgroup  perfect code of $GC(G, S, \alpha)$, then  $\alpha(s)s\notin H$.
Since $\alpha(H)=H$, we have $s\alpha(s)\notin H$, which is equivalent to $H\alpha(s^{-1})\neq Hs$.
\end{remark}

Some basic results about subgroup perfect code are presented as following.

\begin{proposition}\label{pro2.10}
Let $H$ be a subgroup of $G$, $\alpha$ an involution in $\Aut(G)$.
For any element $g\in \Fix_\alpha(G)$, if $H$ is perfect code of $GC(G, S, \alpha)$ for some generalized Cayley subset induced by $\alpha$, then $g^{-1}Hg$ is perfect code of $GC(G, g^{-1}Sg, \alpha)$.
\end{proposition}
\begin{proof}
Let $S=\{s_1, s_2, \cdots, s_r\}$.
Note that
\begin{equation*}
g^{-1}\omega_\alpha(G)g=\alpha(g^{-1})\{\alpha(x^{-1})x\mid x\in G\}g=\{\alpha(xg)^{-1}xg\mid x\in G\}=\omega_\alpha(G).
\end{equation*}
In view of $S\cap \omega_\alpha(G)=\emptyset$, we obtain that $g^{-1}Sg\cap g^{-1}\omega_\alpha(G)g=g^{-1}Sg\cap \omega_\alpha(G)=\emptyset$.
Combining the fact that
\begin{equation*}
\alpha(g^{-1}Sg)=g^{-1}\alpha(S)g=g^{-1}S^{-1}g=(g^{-1}Sg)^{-1},
\end{equation*}
it follows that $g^{-1}Sg$ is a generalized Cayley subset of $G$ induced by $\alpha$.
Since $H$ is perfect code of $GC(G, S, \alpha)$, $\alpha(H)=H$ and $G=H\cup Hs_1\cup \cdots \cup Hs_r$.
Further,  we have
\begin{align*}
G&=g^{-1}Hg\cup g^{-1}Hs_1g\cup \cdots \cup g^{-1}Hs_rg\\
&=g^{-1}Hg\cup (g^{-1}Hg)(g^{-1}s_1g)\cup \cdots \cup (g^{-1}Hg)(g^{-1}s_rg).
\end{align*}
By Corollary \ref{cor1}, $g^{-1}Hg$ is perfect code of $GC(G, g^{-1}Sg, \alpha)$.
\end{proof}

\begin{proposition}\label{pro2.11}
Let $H$ be a subgroup of $G$ and $\alpha$ an involutory automorphism of $G$.
For any $\beta\in \Aut(G)$, if $H$ is perfect code of $GC(G, S, \alpha)$ for some generalized Cayley subset induced by $\alpha$, then $H^\beta$ is perfect code of $GC(G, S^\beta, \alpha^\beta)$, where $\alpha^\beta=\beta\alpha\beta^{-1}$.
\end{proposition}
\begin{proof}
Let $S=\{s_1, s_2, \cdots, s_r\}$.
Observe that $\alpha^\beta$ is an involutory automorphism of $G$ and
\begin{align*}
\beta^{-1}(\omega_{\alpha^\beta}(G))&=\beta^{-1}\{\beta\alpha\beta^{-1}(g^{-1})g\mid g\in G\}\\
&=\{\alpha\beta^{-1}(g^{-1})\beta^{-1}(g)\mid g\in G\}\\
&=\{\alpha(\beta^{-1}(g))^{-1}(\beta^{-1}(g))\mid g\in G\}\\
&=\omega_\alpha(G),
\end{align*}
we have
\begin{equation*}
S^\beta\cap \omega_{\alpha^\beta}(G)=\beta(S\cap \beta^{-1}(\omega_{\alpha^\beta}(G)))=\beta(S\cap \omega_\alpha(G))=\emptyset.
\end{equation*}
Since $\alpha^\beta(S^\beta)=\beta\alpha\beta^{-1}\beta(S)=\beta\alpha(S)
=\beta(S^{-1})=(S^\beta)^{-1}$, $S^\beta$ is generalized Cayley subset of $G$ induced by $\alpha^\beta$.
 As $H$ is perfect code of $GC(G, S, \alpha)$, by Corollary \ref{cor1}, $\alpha(H)=H$, and $G=H\cup Hs_1\cup \cdots \cup Hs_r$.
Thus, $\alpha^\beta(H^\beta)=\beta\alpha(H)=\beta(H)=H^\beta$ and
\begin{align*}
G&=\beta(H)\cup \beta(Hs_1)\cup \cdots \cup \beta(Hs_r)\\
&=H^\beta\cup H^\beta s_1^\beta\cup \cdots \cup H^\beta s_r^\beta.
\end{align*}
Applying Corollary \ref{cor1} again, $H^\beta$ is subgroup perfect code of $GC(G, S^\beta, \alpha^\beta)$.
\end{proof}

Based on the above two propositions, we obtain the following result.

\begin{proposition}\label{pro2.12}
Let $H$ be a subgroup of $G$ and $\alpha$ an involutory automorphism of $G$.
For any $\beta\in \Aut(G)$ and $g\in \Fix_{\alpha^\beta}(G)$, if $H$ is perfect code of $GC(G, S, \alpha)$ for some generalized Cayley subset induced by $\alpha$, then $H^\beta$ is is perfect code of $GC(G, g^{-1}S^\beta g, \alpha^\beta)$.
\end{proposition}

Let $\alpha_1$ be an involutory automorphism of group $G_1$ and $\alpha_2$ an involutory automorphism of group $G_2$.
Let $\bar{G}=G_1\times G_2=\{g_1, g_2\mid g_1\in G_1, g_2\in G_2\}$.
Then $\bar{\alpha}$ which defined by $\bar{\alpha}(g_1,g_2)=(\alpha_1(g_1),\alpha_2(g_2))$ is an involutory automorphism of $G$.

\begin{lemma}\label{lem2.13}
$\omega_{\bar{\alpha}}(\bar{G})=\omega_{\alpha_1}(G_1)\times \omega_{\alpha_2}(G_2)$, $\Omega_{\bar{\alpha}}(\bar{G})=(K_{\alpha_1}(G_1)\times K_{\alpha_2}(G_2))\setminus \omega_{\bar{\alpha}}(\bar{G})$.
\end{lemma}
\begin{proof}
For any $g_1\in G_1$ and $g_2\in G_2$,
\begin{eqnarray*}
\omega_{\bar{\alpha}}(\bar{G})
&=&\{\bar{\alpha}(g_1, g_2)^{-1}(g_1, g_2)\mid g_1\in G_1, g_2\in G_2 \}\\
&=&\{\bar{\alpha}(g_1^{-1})g_1, \alpha(g_2^{-1})g_2 \mid g_1\in G_1, g_2\in G_2\}\\
&=&\omega_{\alpha_1}(G_1)\times \omega_{\alpha_2}(G_2)
\end{eqnarray*}
and
\begin{eqnarray*}
\Omega_{\bar{\alpha}}(\bar{G})
&=&\{(g_1, g_2)\mid \alpha(g_1, g_2)=(g_1^{-1}, g_2^{-1}), ~\text{where}~ g_1\in G_1, g_2\in G_2 ~\text{and} (g_1, g_2)\notin \omega_{\bar{\alpha(G)}}\}\\
&=&(K_{\alpha_1}(G_1)\times K_{\alpha_2}(G_2))\setminus \omega_{\bar{\alpha}}(\bar{G}).
\end{eqnarray*}
\end{proof}

Let  $S_1=\{s_{11}, \cdots, s_{1i}\}$ be generalized Cayley subset of $G_1$  induced by $\alpha_1$ and $S_2=\{s_{21}, \cdots, s_{2j}\}$ a generalized Cayley subset of $G_2$  induced by $\alpha_2$.
Then $T_1=S_1\times S_2$ is generalized Cayley subset of $\bar{G}$ induced by $\bar{\alpha}$.
The details can be  referred for \cite{A.K.P.A}.
In addition, let $T_2=S_1\times S_2 \cup \{e\}\times S_2 \cup S_1\times \{e\}$.
By Lemma \ref{lem2.13},
 $T_2\cap \omega_{\bar{\alpha}}(\bar{G})=\emptyset$ and
\begin{eqnarray*}
\bar{\alpha}(T_2)
&=&\alpha_1(S_1)\times \alpha_2(S_2) \cup \{e\}\times \alpha_2(S_2) \cup \alpha_1(S_1)\times \{e\}\\
&=&S_1^{-1} \times S_2^{-1}\cup \{e\}\times S_2^{-1} \cup S_1^{-1}\times \{e\}\\
&=&(S_1, S_2)^{-1}\cup (\{e\}\times S_2)^{-1}\cup (S_1\times \{e\})^{-1}\\
&=& T_2^{-1}.
\end{eqnarray*}
Thus, $T_2$ is also also a generalized Cayley subset of $\bar{G}$ induced by $\bar{\alpha}$.
Moreover, we find the relationship between the subgroup perfect codes of $G_i(i=1,2)$ and $\bar{G}$.

\begin{proposition}
Let $H_1$ be a subgroup perfect code of $GC(G_1, S_1, \alpha_1)$ and $H_2$  a a subgroup perfect code of $GC(G_2, S_2, \alpha_2)$.
Then $H_1\times H_2$ is subgroup perfect code of $GC(\bar{G}, T_2, \bar{\alpha})$.
\end{proposition}
\begin{proof}
Since $G_1=H_1\cup H_1s_{11}\cup \cdots \cup  H_1s_{1i}$ and $G_2=H_2\cup H_2s_{21}\cup \cdots \cup  H_2s_{2j}$, we have
\begin{eqnarray*}
\bar{G}&=&G_1\times G_2=(H_1\cup H_1s_{11}\cup \cdots \cup  H_1s_{1i})\times (H_2\cup H_2s_{21}\cup \cdots \cup  H_2s_{2j})\\
&=&(H_1\times H_2)\cdot (\{e\}\cup \{e\}\times S_2\cup S_1\times \{e\}\cup S_1\times S_2)\\
&=&(H_1\times H_2)\cdot (\{e\}\cup T_2).
\end{eqnarray*}
By Corollary \ref{cor1}, $H_1\times H_2$ is perfect code of $GC(\bar{G}, T_2, \bar{\alpha})$.
\end{proof}

\begin{proposition}
Let $H_1$ be a subgroup perfect code of $GC(G_1, S_1, \alpha_1)$ and $H_2$  a a subgroup perfect code of $GC(G_2, S_2, \alpha_2)$.
Then $H_1\times H_2$ is  subgroup total perfect code of $GC(\bar{G}, T_1, \bar{\alpha})$.
\end{proposition}

For abelian group, the  equivalent condition of a subgroup to be a perfect code is discussed in the following.

\begin{proposition}\label{pro2}
Let $G$ be an abelian group and $H$ a subgroup of $G$ with odd order.
If there exists an involutory automorphism $\alpha$ such that $H\leq \omega_\alpha(G)$, then $H$ is perfect code of generalized Cayley graph induced by $\alpha$ if and only if $H=\omega_\alpha(G)$.
\end{proposition}
\begin{proof}
Necessity.
Assume $H\subset \omega_\alpha(G)$ and choose $g\in \omega_\alpha(G)\setminus H$.
Then each element in $Hg$ belongs to $\omega_\alpha(G)$.
By Corollary \ref{cor1}, $H$ is not perfect code of any generalized Cayley graph
induced by $\alpha$.

Sufficiency.
Let $H=\omega_\alpha(G)$.
For any nontrivial coset $Hg$, $g\in \Omega_\alpha(G)\cup \mho_\alpha(G)$.
In fact, if $g\in \mho_\alpha(G)$, then $Hg\neq H\alpha(g^{-1})$.
Otherwise, $\alpha(g)g\in H$, this yields that $\alpha(\alpha(g)g)=\alpha(g)g=(\alpha(g)g)^{-1}$.
Thus, $\alpha(g)g$ is an involution in $H$, which is contrary to the order of $H$ is odd.
In addition,  for any $g_1\in \Omega_\alpha(G)$ and $g_2\in \mho_\alpha(G)$,
since $Hg_1=H\alpha(g_1)^{-1}$ and $Hg_2\neq H\alpha(g_2^{-1})$, we have $Hg_1\neq Hg_2$.
Therefore, there exists a set of coset representative of $H$ with form $\{e\}\cup \{s_1, \cdots, s_k\}\cup \{t_1, \alpha(t_1^{-1}), \cdots, t_\ell, \alpha(t_\ell^{-1})\}$,  where $s_i\in \Omega_\alpha(G)$ for $1\leq i\leq k$ and $t_j\in \mho_\alpha(G)$ for $1\leq j\leq \ell$.
$S=\{s_1, \cdots, s_k\}\cup \{t_1, \alpha(t_1^{-1}), \cdots, t_\ell, \alpha(t_\ell^{-1})\}$.
As $\alpha(H)=H$ and $\{e\}\cup S$ is generalized Cayley coset representative of $H$ in $G$, $H$ is perfect code of $GC(G, S, \alpha)$.
\end{proof}

\begin{theorem}\label{thm3}
Let $G$ be an abelian group and $H$ a subgroup of $G$.
Let $\alpha$ be an involutory automorphism of $G$.
Then $H$ is perfect code of generalized Cayley graph induced by $\alpha$ if and only if $\alpha(H)=H$ and for any $g\in G\setminus H$, $\alpha(g)g\in H$ implies there exists an element $h\in H$ such that $hg\in \Omega_\alpha(G)$.
\end{theorem}
\begin{proof}
Assume  $H$ is perfect code of graph $GC(G, S, \alpha)$, where $S=\{s_1, \cdots, s_r\}$.
By Proposition \ref{pro1}, we have $\alpha(H)=H$ directly.
Now assume $\alpha(g)g\in H$ for $g\in G\setminus H$ and $g\in Hs_i$.
Then we obtain $Hs_i=Hg=H\alpha(g^{-1})=H\alpha(s_i^{-1})$.
Recall that both $s_i$ and $\alpha(s_i^{-1})$ belong to $S$ and $\{e\}\cup S$ is a set of right coset representatives of $H$, it follows that
$s_i=\alpha(s_i^{-1})$.
Since $Hg=Hs_i$, there exists $h\in H$ such that $hg=s_i$ and then $\alpha(hg)hg=e$.

Conversely, assume that  $\alpha(H)=H$ and for any $g\in G\setminus H$, $\alpha(g)g\in H$ implies there exists an element $h\in H$ such that $hg\notin \omega_\alpha(G)$ and $\alpha(hg)hg=e$.
Let $\Phi$ denote the set of elements of  $G/H$  not including $H$, and define two subsets of $\Phi$ as follows:
$I=\{Hx_i\mid Hx_i\in \Phi , Hx_i=H\alpha(x_i^{-1})\}$, $J=\Phi\setminus I$.
Assume $I=\{Hx_1, \cdots, Hx_\ell\}$.
For any $Hx_i\in I$, $\alpha(x_i)x_i\in H$, which gives that there exists $h_i\in H$ such that $h_ix_i\notin \omega_\alpha(G)$ and $\alpha(h_ix_i)h_ix_i=e$.
Let $z_i=h_ix_i$, then $\alpha(z_i)=z_i^{-1}$ and $I=\{Hz_1, \cdots, Hz_\ell\}$.
For any $Hy_j\in J$, note that  $H\alpha(y_j^{-1})\neq Hy_j$ and $H\alpha(y_j^{-1})\in J$.
Without loss of generality, assume that $J=\{Hy_1, H\alpha(y_1^{-1}), \cdots, Hy_t, H\alpha(y_t^{-1})\}$.
Let $S=\{z_1, \cdots, z_\ell, y_1, \alpha(y_1^{-1}), \cdots, y_t, \alpha(y_t^{-1})\}$.
It is easy to see $S$ is a generalized Cayley subset induced by $\alpha$ and $\{e\}\cup S$ is a set of right coset representatives of $H$.
By Corollary \ref{cor1}, $H$ is subgroup perfect code of $GC(G, S, \alpha)$.
\end{proof}

\begin{remark}
If $H\leq\omega_\alpha(G)$, the order of $H$ is odd, and $H$ is perfect code induced by $\alpha$, by \ref{thm3}, we obtain $H=\omega_\alpha(G)$.
This follows from if  $H\subset\omega_\alpha(G)$, choose $g\in \omega_\alpha(G)\setminus H$, then $\alpha(g)g=e\in H$. However, for any element $h\in H$,  $hg\in \omega_\alpha(G)$, this is contradiciton.
Conversely, assume $H=\omega_\alpha(G)$, then for any element $g\in G$ satisfying $\alpha(g)g\in H$, $\alpha(g)g=(\alpha(g)g)^{-1}$ holds, that is, $(\alpha(g)g)^2=e$.
As  a result, $\alpha(g)g=e$, this implies that $h\in \Omega_\alpha(G)$.
By \ref{thm3} again, $H$ is perfect code induced  by $\alpha$.
Therefore,  Proposition \ref{pro2} can be viewed as a corollary of Theorem \ref{thm3}.
\end{remark}

Based on Theorem \ref{thm3}, there are several corollaries.

\begin{corollary}\label{cor2}
Let $G$ be an abelian group and $H$ is a characteristic subgroup of $G$.
If $H\subseteq \omega_\alpha(G)$, then $H$ is perfect code with respect to $\alpha$ if and only if $H=\omega_\alpha(G)$ and for any $g\in G$, $\alpha(g)g\in H$ implies $\alpha(g)g=e$.
\end{corollary}

\begin{corollary}\label{cor4}
Let $G$ be an abelian group and $H$ is a  subgroup of $G$ with  odd order.
Then $H$ is subgroup perfect code of $G$.
In particular, if $H$ is perfect code of generalized Cayley graph induced by $\iota$, then $H=\omega_\alpha(G)$.
\end{corollary}
\begin{proof}
For the involutory automorphism $\iota$, $\omega_\iota(G)=\{g^2\mid g\in G\}$.
Since the order of $H$ is odd, we have $H=\{h^2\mid h\in H\}$ and then $H$ is contained in $\omega_\iota(G)$.
By Proposition \ref{pro2}, it follows that $H=\omega_\alpha(G)$.
\end{proof}

\begin{proposition}\label{pro5}
Let $\alpha$ be an involutory automorphism of finite group $G$ and $H$ a subgroup perfect code of generalized Cayley graph induced by $\alpha$.
If $H\leq K\leq G$ and $\alpha(K)=K$, then $H$ is a subgroup perfect code of $K$ with respect  to  $\alpha$.
\end{proposition}
\begin{proof}
Assume that $H$ a subgroup perfect code of graph $GC(G, S, \alpha)$.
Then $\{e\}\cup S$ is a set of coset representatives of $H$.
Let $S_1=S\cap K$.
Then $\alpha(S_1^{-1})=\alpha(S^{-1})\cap \alpha(K^{-1})=S\cap K=S_1$, which implies that $S_1$ is  a generalized Cayley subset of $K$ induced by $\alpha$.
Combining the fact that  $\{e\}\cap S_1$ is a set of coset representatives of $H$ in $K$, by Corollary \ref{cor1}, the result follows.
\end{proof}

Observe that  if $\alpha(H)=H$, then for  any $g\in N_G(H)$, $\alpha(g)H=\alpha(gH)=\alpha(Hg)=H\alpha(g)$, which implies $\alpha(N_G(H))=N_G(H)$.
In Proposition \ref{pro5}, let $K=N_G(H)$, then we have the following corollary.

\begin{corollary}
Let $\alpha$ be an involutory automorphism of finite group $G$ and $H$ a subgroup perfect code of generalized Cayley graph induced by $\alpha$.
Then $H$ is subgroup perfect code of $N_G(H)$ with respect  to  $\alpha$.
\end{corollary}

\section{Total perfect code}

From the discussion of perfect code in Section 2, we can easily obtain the following result about total perfect code in generalized Cayley graphs.

\begin{proposition}
Let $S=\{s_1, \cdots, s_r\}$ be a generalized Cayley subset of $G$ induced by the involutory automorphism $\alpha$ and $X$ a subset of $G$.
The following conditions are equivalent:
\begin{enumerate}[{\rm(i)}]

\item $X$ is a total perfect code of $GC(G, S, \alpha)$;

\item $\{\alpha(X)s_1, \cdots, \alpha(X)s_r\}$ is a partition of $G$;

\item $|G|=|X|r$ and $(\alpha(X^{-1})\alpha(X))\cap(SS^{-1})=\emptyset$.

\end{enumerate}
\end{proposition}

If $H$ is a subgroup of $G$, then $H$ is total perfect code of the generalized Cayley graph $GC(G, S, \alpha)$ if and only if $S$ is right transversal of $\alpha(H)$ in $G$.
Note that
\begin{eqnarray*}
G
&=&\alpha(H)s_1\cup \alpha(H)s_2\cup \cdots \cup \alpha(H)s_r\\
&=&\alpha(Hs_1^{-1})\cup \alpha(Hs_2^{-1})\cup \cdots \cup \alpha(Hs_r^{-1})\\
&=&\alpha(s_1)\alpha(H)\cup \alpha(s_2)\alpha(H)\cup \cdots \cup \alpha(s_r)\alpha(H)\\
&=&s_1^{-1}\alpha(H)\cup s_2^{-1}\alpha(H)\cup \cdots \cup s_r^{-1}\alpha(H),
\end{eqnarray*}
thus the generalized Cayley subset $S$ is the right coset transversal of $\alpha(H)$ if and only if $S^{-1}$ is the left coset transversal of $\alpha(H)$.

It can be checked that  Propositions \ref{pro2.10} to \ref{pro2.12} also hold for subgroup total perfect  code.
Unlike subgroup perfect code,  the subgroup $H$ of $G$ is  total perfect code of $GC(G, S, \alpha)$ does not imply $\alpha(H)=H$.

\end{sloppypar}

\end{document}